\documentclass[11pt]{article}

\usepackage{amsmath}
\usepackage{amsfonts}
\usepackage{amssymb}
\usepackage{mathabx}
\usepackage[all,2cell]{xy}
\UseTwocells

\newtheorem{definition}{Definition}

\newtheorem{proposition}{Proposition}

\newtheorem{lemma}{Lemma}

\newenvironment{proof}{{\bf Proof:}}{$\text{ }\blacksquare$}

\begin{document}

\title{Higher dimensional operads}
\author{Dennis Borisov\\ Max-Planck Institute for Mathematics, Bonn, Germany\\ dennis.borisov@gmail.com}
\date{\today}
\maketitle

\begin{abstract}
The theory of operads (May, cyclic, modular, PROPs, etc) is extended to include higher dimensional phenomena, i.e. operations between operations, mimicking the algebraic structure on varieties of arbitrary dimensions, having marked subvarieties of arbitrary codimension.
\end{abstract}

\section{Introduction}

Operads (as defined by P.May in \cite{Ma72}) can be described as formalizations of properties of $\{Hom(V^{\otimes^m,},V)\}_{m\geq 0}$, for a generic vector space $V$. Similarly, PROPs are modeled on $\{Hom(V^{\otimes^m},V^{\otimes^n})\}_{m,n\geq 0}$.

This kind of models is not suitable to describe cyclic operads, which are modeled on $\{V^{n+1}\}_{n\geq 0}$, using a bilinear form on $V$.  In turn, this model is not suitable anymore to describe modular operads.

\smallskip

There is one model, that is good for most kinds of operads: real (singular) curves with marked points. Here we consider every curve as an operation, connecting the marked points on it. Having several such curves we can glue them at marked points, or take their disjoint union. 

In some contexts we might consider unoriented curves, and hence arrive at cyclic and modular operads, in other contexts we can orient the curves and obtain PROPs, or more generally wheeled PROPs. This approach to operads was developed in \cite{KM94}, \cite{GK98}, \cite{BeM96}, \cite{BM08}.

\smallskip

All these different theories have one thing in common: they are modeled on $1$-dimensional objects -- curves. This paper deals with the following problem: find the formalization of the algebraic properties of higher dimensional varieties, when we have marked subvarieties of arbitrary codimension, perhaps having non-empty intersections. 

These marked subvarieties can have their own marked sub-subvarieties, and so on. Gluing happens on many levels at once, e.g. if a marked point $P$ sits on two marked curves $C_1$, $C_2$ in a variety, gluing this variety with another one along $C_1$ is accompanied by gluing $C_2$ with something at $P$. 

In this paper we solve this problem, and call the resulting concept {\it nested operads}.

\medskip

To define $1$-dimensional operads using curves, one uses the dual graphs of curves, i.e. if we have a curve $C$ with marked points $\{p_1,\ldots,p_n\}$, the dual graph has one vertex, corresponding to $C$, and $n$ flags coming out of it (corresponding to the marked points). If $p_i=p_j$, the corresponding flags comprise an edge, and so on.

In effect we have two kinds of nodes in dual graphs: vertices, that correspond to curves, and ends of flags, that correspond to points. In the higher dimensional case we obviously have more than two kinds. If we want to define dual graphs here, we get graphs, whose flags are graphs on their own, whose flags are also graphs, and so on. This explains the terms {\it nested} graphs and nested operads.

\smallskip

The following observation gives a simple way to define such ``nested'' graphs: if a variety $L$ has a marked subvariety $M$, which in turn has a marked subvariety $N$, then $N$ is also a marked subvariety of $L$; on the other hand it cannot happen that $L$ is a marked subvariety of $M$.

The transitivity property of inclusion lets us model nested graphs as categories: objects correspond to subvarieties, morphisms to inclusions. The requirement that we do not have cycles implies that these categories have to be direct, i.e. without non trivial composable sequences of morphisms, starting and ending at the same point.

\smallskip

As with the dual graphs of curves, the hard part in defining the corresponding theory of operads is not in the definition of graphs, but in the definition of morphisms between graphs (see \cite{BM08} for the $1$-dimensional case). 

Defining the category of nested graphs is one of the main goals of this paper. Actually it is a double category, reflecting the fact that we should not only know how to glue varieties, but also how to embed one into another. 

\smallskip

The other goal is to describe the additional structure on the resulting category, that allows us to define nested operads. In the $1$-dimensional case it is a monoidal structure, given by disjoint union, and it lets us define operads of all kinds as representations of the category of graphs, i.e. as monoidal functors $\Gamma\rightarrow\mathcal G$, where $\Gamma$ is the category of labeled graphs, that we chose, and $\mathcal G$ is any monoidal category (see \cite{BM08} for details).

In the nested case, the structure is more complicated. Given two varieties $U$ and $V$, we can glue them along $0$-dimensional subvarieties, or $1$-dimensional, $2$-dimensional and so on. Having just one monoidal product on the category of nested graphs is not enough. 

Instead of several monoidal products we use one action of the club of simplicial sets $\underline{sset}$. The corresponding definitions are given in \cite{Bo10}. This allows us to define nested operads as representations of categories of labeled nested graphs, i.e. as functors $N\Gamma\rightarrow\mathcal G$, where $\mathcal G$ is also a $\underline{sset}$-algebra, and the functor respects the action of $\underline{sset}$.


\section{The category of nested graphs}

\begin{definition}\label{DefinitionGraph} {\bf A nested graph} is a small category $\mathcal N$, s.t. there is at least one functor ({\bf a grading})
	$$G:\mathcal N\rightarrow\underline{n},$$
where $\underline{n}$ is a finite ordinal\footnote{We consider $\underline{n}$ as a category where $Hom(i,j)$ is empty if $i>j$, and $Hom(i,j)=pt$ otherwise.} and $G$ maps non-identity morphisms to non-identity morphisms.\end{definition} 

To simplify our dealings with nested graphs we introduce some terminology in the following definition. It is obviously modeled on the terminology, used in the theory of usual graphs. 

To make the connection we would like to note that $1$-dimensional graphs are a special kind of nested ones, namely those that have a grading into the ordinal $\underline{2}=\{0,1\}$. Here vertices are objects in $G^{-1}(1)$, flags are the non-identity morphisms, and two flags build up an edge (or a hyper-edge for more than two), if the corresponding morphisms have the same domain.

\begin{definition}\label{GraphKinds}\begin{itemize}
\item {\bf A flag} in a nested graph is a non-identity morphism. {\bf An irreducible flag} is a flag, that cannot be written as a composition of two flags.
\item Objects in a nested graph will be called {\bf nodes}. We consider every flag as being {\bf decorated} by its domain. If a node is the codomain of a flag, we will say that {\bf the flag is attached to the node}.
\item {\bf A vertex} in a graph is a node, that does not decorate any flag. {\bf A corolla} is a nested graph with a unique vertex.
\item A functor $\phi:\mathcal N_1\rightarrow\mathcal N_2$ {\bf contracts a flag} $f\in Mor(\mathcal N_1)$ if $\phi(f)$ is an identity. {\bf A node $A$ is contracted by a functor}, if it decorates a flag, that is contracted.\end{itemize}\end{definition}
It is clear that in a $1$-dimensional graph all flags are irreducible. In dimension higher than $1$ we obviously have reducible flags, e.g. a marked point belonging to a surface can also belong to a marked curve on that surface.

We will define morphisms between nested graphs using functors between the categories. However, not all functors will be allowed. 

\begin{definition} A functor $\phi:\mathcal N_1\rightarrow\mathcal N_2$ will be called {\bf admissible} if:\begin{itemize}
\item[1.] for any irreducible flag $f\in\mathcal N_1$, $\phi(f)$ is either identity or an irreducible flag,
\item[2.] for any irreducible flag $f:A\rightarrow B$ in $\mathcal N_1$, if $\phi$ contracts $f$, it contracts all irreducible flags, decorated by $A$. \end{itemize}\end{definition}
Geometric meaning of these conditions is obvious: \begin{itemize}
\item[1.] If we have a variety $U$, and a marked subvariety $V\subset U$, s.t. there is no marked subvariety in between, then after gluing $U$ with anything, this property remains, i.e. gluing cannot insert anything between $U$ and $V$.
\item[2.] If we glue several varieties $\{U_i\}$ along a marked subvariety $V\subset U_i$, then after gluing, $V$ ceases to be marked in all of $U_i$'s. In other words, when we contract an edge (or a hyper-edge), we contract all flags involved.
\end{itemize}

The following lemma is immediate.

\begin{lemma} Composition of two admissible functors is admissible.\end{lemma}

In \cite{BM08} it is shown that any morphism between $1$-dimensional graphs can be uniquely written as a composition of a grafting, followed by  merger, followed by a contraction. 

In the higher dimensional case there is no distinction between grafting and merger, since now both ends of flags are just objects in a category. Therefore a morphism between nested graphs will be defined as a merger, followed by a contraction.

\begin{definition}\label{DefinitionMerCon}\begin{itemize}\item[1.] {\bf A merger} $\mu:\mathcal N_1\rightarrow\mathcal N_2$ is an admissible epi-functor,\footnote{By {\bf an epi-functor} we mean $\mathcal N_1\rightarrow\mathcal N_2$, s.t. image of $\mathcal N_1$ generates all of $\mathcal N_2$.} s.t. $\mathcal N_2$ is a quotient of $\mathcal N_1$ by an equivalence relation on objects of $\mathcal N_1$.
\item[2.] {\bf A contraction} $\kappa:\mathcal N_1\rightarrow\mathcal N_2$ is an admissible epi-functor, s.t. for every node $A_2\in\mathcal N_2$, the fiber $\kappa^{-1}(A_2)$ is a corolla.\end{itemize}\end{definition} 

\begin{lemma} Composition of mergers is a merger, composition of contractions is a contraction.\end{lemma}
\begin{proof} It is obvious that compositions of mergers are mergers. To prove the statement for contractions, consider a contraction $\kappa:\mathcal N_1\rightarrow\mathcal N_2$, let $A_2\in\mathcal N_2$ be any node, and let $A_1\in\mathcal N_1$ be the only vertex in $\kappa^{-1}(A_2)$. We claim that for any flag $f_2:A_2\rightarrow B_2$ in $\mathcal N_2$ there is a flag $f_1:A_1\rightarrow B_1$ in $\mathcal N_1$, s.t. $\kappa(f_1)$ divides $f_2$.

Indeed, since $\kappa(\mathcal N_1)$ generates all of $\mathcal N_2$, there is a flag $g_1\in\mathcal N_1$, that is not contracted by $\kappa$, s.t. $\kappa(g_1)$ is decorated by $A_2$ and $\kappa(g_1)$ divides $f_2$. Moreover, we can choose $g_1$ to be irreducible. Then, since $\kappa$ is admissible, the domain of $g_1$ cannot be contracted by $\kappa$, and therefore $g_1$ has to be decorated by $A_1$.

Now let $\mathcal N_1\overset{\kappa_1}\rightarrow\mathcal N_2\overset{\kappa_2}\rightarrow\mathcal N_3$ be two contractions, and let $A_3\in\mathcal N_3$ be any node. Let $A_2\in\mathcal N_2$ be the only vertex in $\kappa_2^{-1}(A_3)$, and let $A_1$ be the only vertex in $\kappa_1^{-1}(A_2)$. We claim that $A_1$ is the only vertex in $\kappa_1^{-1}(\kappa_2^{-1}(A_3))$.

Indeed, Let $B_2\in\kappa_2^{-1}(A_3)$ be any node different from $A_2$. Since $A_2$ is the unique vertex in $\kappa_2^{-1}(A_3)$, there is a flag $f_2\in\mathcal N_2$, decorated by $B_2$, s.t. it is contracted by $\kappa_2$. Let $B_1$ be the unique vertex in $\kappa_1^{-1}(B_2)$. As we have seen above, there is a flag $f_1\in\mathcal N_1$, decorated by $B_1$, s.t. $\kappa_1(f_1)$ divides $f_2$. Then obviously $f_1$ is contracted by $\kappa_2\circ\kappa_1$, and therefore $B_1$ cannot be a vertex in $\kappa_1^{-1}(\kappa_2^{-1}(A_3))$.\end{proof}

\smallskip

The previous lemma tells us how to compose mergers with mergers and contractions with contractions. To define compositions of pairs (contraction, merger), we need to know what is a contraction, followed by a merger. The following lemma provides an answer.

\begin{lemma}\label{LemmaDecomposition} Any admissible epi-functor between nested graphs can be written (in a non-unique way) as a composition of a merger, followed by a contraction.\end{lemma}
\begin{proof} We will not only prove that this is possible, we will construct canonical decompositions, that will be used later to organize nested graphs into a category. We do it by requiring the merger to be as small as possible. 

Let $\phi:\mathcal N_1\rightarrow\mathcal N_2$ be an admissible epi-functor. It defines an equivalence relation on $Obj(\mathcal N_1)$ as follows: for $A_1\neq B_1$, $A_1\sim B_1$ iff $\phi(A_1)=\phi(B_1)$ and both $A_1$ and $B_1$ are vertices in $\phi^{-1}(\phi(A_1))$. Let $\mathcal N'_2$ be the category, obtained by identifying equivalent nodes in $\mathcal N_1$, and morphisms being freely generated by morphisms in $\mathcal N_1$, subject to relations of composition in $\mathcal N_1$.

In addition to projection $\mu:\mathcal N_1\rightarrow\mathcal N'_2$, there is an obvious $\kappa:\mathcal N'_2\rightarrow\mathcal N_2$, s.t. $\phi=\kappa\circ\mu$. It is clear that both $\mu$ and $\kappa$ are epi-functors. We claim that $\mathcal N'_2$ is a nested graph, $\mu$ is a merger, and $\kappa$ is a contraction.

Suppose $\mathcal N'_2$ is not a nested graph, i.e. there is a cycle of flags $\overset{f^{(n)}_2}\rightarrow\dots\overset{f'_2}\rightarrow$. Since $\mu$ is an epi-functor, we can assume that $\forall i\quad f_2^{(i)}=\mu(f^{(i)}_1)$, for some $f^{(i)}_1:A^{(i)}_1\rightarrow B^{(i)}_1$ in $\mathcal N_1$. Since $\mathcal N_1$ is a direct category there is an $i$, s.t. $B^{(i)}_1\neq A^{(i+1)}_1$. However, $\mu(B^{(i)}_1)=\mu(A^{(i+1)}_1)$, because $f^{(i)}_2$ and $f^{(i+1)}_2$ are composable. Then, by definition of $\mu$, $B^{(i)}_1$ and $A^{(i+1)}_1$ are vertices in the same fiber of $\phi$, but $f^{(i+1)}_1$ is contracted by $\phi$, since $f^{(i+1)}_2$ is contracted by $\kappa$ ($\mathcal N_2$ is a direct category), so $A^{(i+1)}_1$ cannot be a vertex -- contradiction.

Since $\mu$ is defined by identifying nodes, it is clearly a merger. Let $A_2\in\mathcal N_2$ be a node. The fiber $\kappa^{-1}(A_2)$ has a unique vertex, because it is obtained by identifying all vertices in $\phi^{-1}(A_2)$.

It remains to prove that $\kappa$ is an admissible functor. It is immediately obvious that irreducible flags in $\mathcal N'_2$ are precisely the images under $\mu$ of irreducible flags in $\mathcal N_1$. Therefore, it is clear that $\kappa$ maps irreducible flags to irreducible flags. 

Let $f_2:A_2\rightarrow B_2$ be an irreducible flag, that is contracted by $\kappa$. Then there is an irreducible flag $f_1:A_1\rightarrow B_1$ in $\mathcal N_1$, contracted by $\phi$, and s.t. $\mu(f_1)=f_2$. Clearly $A_1$ is a not a vertex in a fiber of $\phi$, and therefore, by definition of $\mu$, $\mu^{-1}(A_2)=A_1$, and consequently $\kappa$ contracts all irreducible flags, decorated by $A_2$.\end{proof}

\begin{definition} {\bf The category of nested graphs} $NGr$ is defined as follows:\begin{itemize}
\item[1.] Objects are nested graphs, as in Definition \ref{DefinitionGraph}.
\item[2.] Morphisms from $\mathcal N_1$ to $\mathcal N_2$ are pairs $\mathcal N_1\overset{\mu}\rightarrow\overset{\kappa}\rightarrow\mathcal N_2$,\footnote{To be precise, we mean here {\it equivalence classes of pairs}, where two pairs are equivalent if one can be obtained from the other by shifting an isomorphism from the merger to the contraction.} where $\mu$ is a merger and $\kappa$ is a contraction, as in Definition \ref{DefinitionMerCon}.
\item[3.] Composition $(\kappa_2,\mu_2)\circ(\kappa_1,\mu_1)$ is $(\kappa_2\kappa',\mu'\mu_1)$, where $(\kappa',\mu')$ is the decomposition (according to Lemma \ref{LemmaDecomposition}) of $\mu_2\kappa_1$.\end{itemize}\end{definition}

\begin{proposition} Defined as above, $NGr$ is a category.\end{proposition}
\begin{proof} Let $\xymatrix{\mathcal N_1\ar[r]^{(\kappa_1,\mu_1)} & \mathcal N_2\ar[r]^{(\kappa_2,\mu_2)} & \mathcal N_3\ar[r]^{(\kappa_3,\mu_3)} & \mathcal N_4}$ be $3$ morphisms. We would like to prove that the two possible compositions coincide. Since composition of mergers is just composition of functors, it is easy to see that we can assume $\mu_1$ to be trivial. Similarly we can assume $\kappa_3$ being trivial. Therefore, all we have to prove is that
	\begin{equation}\label{Associativity} \mu_3\circ((\kappa_2,\mu_2)\circ\kappa_1)=(\mu_3\circ(\kappa_2,\mu_2))\circ\kappa_1.\end{equation}
	
We will start with a simplification: let $\kappa$ be a contraction, and let $\mu_1$, $\mu_2$ be mergers. We claim 
	\begin{equation}\label{Mergers}(\mu_2\mu_1)\circ\kappa=\mu_2\circ(\mu_1\circ\kappa).\end{equation}
Let's denote $(\mu_2\mu_1)\circ\kappa$ by $(\kappa',(\mu_2\mu_1)')$, and $\mu_2\circ(\mu_1\circ\kappa)$ by $(\kappa'',\mu'_2\mu'_1)$. In construction of $(\mu_2\mu_1)'$ we identify vertices in $\kappa$-fibers over nodes, that are identified by $\mu_2\mu_1$; on the right hand side this identification is split in two steps -- first we identify according to $\mu_1$, and then $\mu_2$. Clearly $(\mu_2\mu_1)'=\mu'_2\mu'_1$, and hence (\ref{Mergers}) is true, since any decomposition of a given admissible epi-functor into a merger, followed by a contraction, is determined by the merger.

Equation (\ref{Mergers}) implies that in (\ref{Associativity}) we can assume $\mu_2$ to be trivial. And hence the problem is reduced to proving that for any contractions $\kappa_1,\kappa_2$ and a merger $\mu$ we have
	\begin{equation}\label{Contractions}\mu\circ(\kappa_2\kappa_1)=(\mu\circ\kappa_2)\circ\kappa_1.\end{equation}
Let $A$ be a node in the domain of $\mu$, let $B$ be the vertex in $\kappa_2^{-1}(A)$. The fiber $(\kappa_2\kappa_1)^{-1}(A)$ is a corolla, its vertex is the vertex $C$ in $\kappa_1^{-1}(B)$. If $\mu$ identifies $A$ and $A'$, then the merger $\mu'$, obtained by commuting $\mu$ and $\kappa_2\kappa_1$, identifies $C$ and $C'$. 

On the right hand side we have two steps: first we identify $B$ and $B'$, and then $C$ and $C'$. The result is clearly the same, and so we obtain the same merger on both sides of (\ref{Contractions}), and therefore the same contraction. 

Finally, identities in $NGr$ are obviously pairs where both the merger and the contraction are the identity functors.\end{proof}

\smallskip

It is easy to check that the full subcategory of $NGr$, consisting of nested graphs that have a grading into $\underline{2}$, is precisely the category of usual (hyper-) graphs. If we impose an additional condition that there are at most two morphisms sharing a domain (i.e. if we allow only edges, but no hyper-edges), we obtain the category $Gr$ from \cite{BM08}.

\bigskip

In the next section we introduce an action of the club of simplicial sets on the category $NGr$. To do this we will need to keep track of inclusions of graphs into other graphs (this is to define the ``amalgamation'' in the ``amalgamated'' products). 

So here we extend the category structure on $NGr$ to a double category. This is also present in the $1$-dimensional case $Gr$, but is not used there, since one can get along without the action of simplicial sets.

\begin{definition}\label{Subgraph}{\bf A full subgraph of a nested graph} $\mathcal N$ is a subcategory $\mathcal N'\subseteq\mathcal N$, s.t. if a node belongs to $\mathcal N'$, so do all the flags attached to it.\end{definition} 

\begin{definition}\label{CategoryO}\begin{itemize}
\item[1.]{\bf The category of objects $\mathfrak O$} has nested graphs as objects, and given two nested graphs $\mathcal M,\mathcal N$, {\bf a dependency} $\xymatrix{\mathcal M\ar@{~>}[r] & \mathcal N}$ is an injective functor $\delta:\mathcal M\rightarrow\mathcal N$, s.t. $\delta(\mathcal M)$ is a full subgraph of $\mathcal N$.
\item[2.] {\bf The category of morphisms $\mathfrak M$} has morphisms between nested graphs as objects, and its morphisms are commutative diagrams
	$$\xymatrix{\mathcal N_1\ar[rr]^\mu\ar@{<~}[d]_{\delta_1} && \mathcal N_2\ar[rr]^\kappa\ar@{<~}[d]_{\delta_2} && 
	\mathcal N_3\ar@{<~}[d]_{\delta_3}\\ \mathcal M_1\ar[rr]_{\mu'} && \mathcal M_2\ar[rr]_{\kappa'} && \mathcal M_3}$$
s.t. 
	\begin{equation}\label{CondMorDep}\kappa^{-1}(\delta_3(\mathcal M_3))=\delta_2(\mathcal M_2),\qquad
	\mu^{-1}(\delta_2(\mathcal M_2))=\delta_1(\mathcal N_1).\end{equation}
\end{itemize}\end{definition}

It is straightforward to check that compositions of morphisms are compatible with the source, domain and the identity functors, i.e. that we have a double category structure on $NGr$. 


\section{Action of simplicial sets}\label{Algebras}

In \cite{Bo10}, it is shown that the category of simplicial sets $\underline{SSet}$ has the structure of a club. Moreover, it is shown that the subcategory $\underline{sset}$, consisting of injective maps, is also a club (with respect to a restricted semi-direct product).

Here we extend the monoidal structure on $Gr$, used in \cite{BM08}, to a full action of $\underline{sset}$. At the end of the section we use this fact to define nested operads in any $\underline{sset}$-algebra.

\begin{proposition} There is an action of $\underline{sset}$ on $\mathfrak O$, $\mathfrak M$, compatible with the double category structure.\end{proposition}
\begin{proof} It is easy to see that corollas are generators for the category $\mathfrak O$, and that morphisms with corollas as codomains are generators for $\mathfrak M$. Using these generators we define $\{\underline{sset}^k(\mathfrak O)\}_{k\geq 0}$, $\{\underline{sset}^k(\mathfrak M)\}_{k\geq 0}$. Since $\iota:\mathfrak O\rightarrow\mathfrak M$ and $\tau:\mathfrak M\rightarrow\mathfrak O$ map generators to generators, we have $\underline{sset}^k(\iota):\underline{sset}^k(\mathfrak O)\rightarrow\underline{sset}^k(\mathfrak M)$, $\underline{sset}^k(\tau):\underline{sset}^k(\mathfrak M)\rightarrow\underline{sset}^k(\mathfrak O)$ for all $k\geq 0$. It is straightforward to check that conditions (\ref{CondMorDep}) imply the same for $\sigma$. 

Now we define functors $\gamma_\mathfrak O:\underline{sset}(\mathfrak O)\rightarrow\mathfrak O$ and $\gamma_\mathfrak M:\underline{sset}(\mathfrak M)\rightarrow\mathfrak M$. Here we use the canonical action of $SSet$ on $Cat$, given by taking colimits. We have to prove that, taking a colimit of a simplicial diagram in $\mathfrak O$, we get a nested graph as the result. This becomes obvious if we choose a common grading on each member of the diagram, and hence the colimit becomes graded itself.

Given a simplicial diagram in $\mathfrak M$, we can write it as a composition of a diagram of mergers and a diagram of contractions. Taking colimits separately, we get three nested graphs and two functors between them. While it is straightforward to see that colimit of mergers is a merger, we have to prove that colimit of contractions is a contraction.

First we show that colimit of admissible functors is admissible. Let $\overline{\mathcal N}$ be an object in $\underline{sset}(\mathfrak O)$, and let $\mathcal N$ be its colimit (in $Cat$). It is easy to see that irreducible flags in $\mathcal N$ are given as equivalence classes of irreducible flags in members of $\overline{\mathcal N}$. Therefore, if $\overline{\phi}:\overline{\mathcal N}\rightarrow\overline{\mathcal N'}$ consists of admissible functors, the corresponding $\phi:\mathcal N\rightarrow\mathcal N'$ maps irreducible flags to irreducible flags. 

Next we make the following observation: let $A$ be a node in $\mathcal N$, if $A$ is a vertex in a $\phi$-fiber, then any pre-image of $A$ with respect to $\overline{\mathcal N}\rightarrow\mathcal N$ is a also a vertex in its $\overline{\phi}$-fiber. Moreover, conditions (\ref{CondMorDep}) imply that this is necessary and sufficient, i.e. if $A$ is not a $\phi$-vertex, none of its pre-images in $\overline{\mathcal N}$ is a $\overline{\phi}$-vertex. It immediately follows that $\phi$ is admissible, if every component of $\overline{\phi}$ is admissible.

Now suppose that every component of $\overline{\phi}$ is a contraction, but $\phi$ is not a contraction, i.e. there is a node $A\in\mathcal N'$, s.t. $\phi^{-1}(A)$ has at least two vertices, call them $B_1$ and $B_2$. Let $C_1$, $C_2$ be any pre-images of $B_1$, $B_2$ in $\overline{\mathcal N}$. Clearly $\overline{\phi}(C_1)$, $\overline{\phi}(C_2)$ have to be glued together in $\mathcal N'$. This gluing is performed by a sequence of nodes in $\overline{\mathcal N'}$, and taking pre-images of these nodes in $\overline{N}$, we conclude that $C_1$, $C_2$ have to be glued together as well. 

So far our construction was defined only for objects of $\underline{sset}(\mathfrak O)$ and $\underline{sset}(\mathfrak M)$. To extend it to functors we need to show that colimits of morphisms in $\underline{sset}(\mathfrak O)$, $\underline{sset}(\mathfrak M)$ are morphisms in $\mathfrak O$, $\mathfrak M$ respectively. This is a direct consequence of the definition of morphisms in $\underline{sset}(\mathfrak O)$, $\underline{sset}(\mathfrak M)$, i.e. them being fibrations (\cite{Bo10}).

Finally, associativity of $\gamma_\mathfrak O$, $\gamma_\mathfrak M$ follows immediately from associativity of the canonical action of $SSet$ on $Cat$.\end{proof}

\begin{definition} {\bf An abstract category of labeled nested graphs} is a double category $N\Gamma$, having a compatible (partial) action of $\underline{sset}$, and a double functor $\psi:N\Gamma\rightarrow NGr$, preserving the action of $\underline{sset}$, s.t. $\psi^{-1}(NGr)$ generates $N\Gamma$ as a $\underline{sset}$-algebra.

Given any double category $\mathcal G$, with an action of $\underline{sset}$, {\bf an $N\Gamma$-operad in $\mathcal G$} is a double functor $N\Gamma\rightarrow\mathcal G$, with contravariant components, preserving the action of $\underline{sset}$.\end{definition}

An example of a double category $\mathcal G$ having an action of $\underline{sset}$ is any category $\mathcal G$ with limits, considered as a double category, with the category of objects being $\mathcal G$ itself, and the category of morphisms being $Mor(\mathcal G)$ (i.e. objects are morphisms in $\mathcal G$, morphisms are commutative squares in $\mathcal G$). 

In this way we can recover the globular operads from \cite{Ba98} as a particular case of our construction.

\end{document}